\documentclass[a4paper,12pt]{article}
\usepackage[utf8]{inputenc}

\usepackage{graphicx}
\usepackage{geometry}
\usepackage{amsthm}
\usepackage{amssymb}
\usepackage{amsmath}
\usepackage{xcolor}
\usepackage{enumerate}
\usepackage{listings}
\usepackage{microtype}
\usepackage{hyperref}
\usepackage{thmtools} 
\usepackage{thm-restate}
\usepackage[normalem]{ulem}
\usepackage[numbers]{natbib}

\graphicspath{{figs/}}

\newtheorem{theorem}{Theorem}

\newtheorem{lemma}[theorem]{Lemma}
\newtheorem{proposition}[theorem]{Proposition}

\newtheorem{problem}{Problem}

\def\inst#1{$^{#1}$}

\date{}

\title{Faces in rectilinear drawings of complete graphs\thanks{
M.\ Balko was supported by grant no.\ 23-04949X of the Czech Science Foundation (GA\v{C}R) and by the Center for Foundations of Modern Computer Science (Charles University project UNCE 24/SCI/008).
A.\ Brötzner was supported by grant 2021-03810 (Illuminate: provably good algorithms for guarding problems) from the Swedish Research Council (Vetenskapsr\r{a}det).
J.\ Tkadlec was supported by Charles Univ. projects PRIMUS 24/SCI/012 and UNCE 24/SCI/008.
}
}
\author{
Martin Balko\inst{1}
\and
Anna Br\"otzner\inst{2}
\and
Fabian Klute\inst{3}
\and
Josef Tkadlec\inst{4}
}

\begin{document}

\maketitle

\begin{center}
{\footnotesize
\inst{1} 
Department of Applied Mathematics, \\
Faculty of Mathematics and Physics, Charles University, Czech Republic \\
\texttt{balko@kam.mff.cuni.cz}\\
\inst{2}
Faculty of Technology and Society, \\
Malm\"o University, Sweden\\
\texttt{anna.brotzner@mau.se}\\
\inst{3}
Departament de Matemàtiques\\
Universitat Politècnica de Catalunya, Barcelona, Spain\\
\texttt{fabian.klute@upc.edu}\\
\inst{4}
Computer Science Institute, \\
Faculty of Mathematics and Physics, Charles University, Czech Republic \\
\texttt{josef.tkadlec@iuuk.mff.cuni.cz}
}\\
\end{center}

\begin{abstract}
We initiate the study of extremal problems about faces in \emph{convex rectilinear drawings} of~$K_n$, that is, drawings where vertices are represented by points in the plane in convex position and edges by line segments between the points representing the end-vertices.
We show that if a convex rectilinear drawing of $K_n$ does not contain a common interior point of at least three edges, then there is always a face forming a convex 5-gon while there are such drawings without any face forming a convex $k$-gon with $k \geq 6$.

A convex rectilinear drawing of $K_n$ is \emph{regular} if its vertices correspond to vertices of a regular convex $n$-gon.
We characterize positive integers $n$ for which regular drawings of $K_n$ contain a face forming a convex 5-gon.

To our knowledge, this type of problems has not been considered in the literature before and so we also pose several new natural open problems.
\end{abstract}

\section{Introduction}

Let $G$ be a graph with no loops nor multiple edges.
In a \emph{rectilinear drawing of $G$} the vertices are represented by distinct points in the plane and each edge corresponds to a line segment connecting the images of its end-vertices.
We consider only drawings where no three points representing vertices lie on a common line.
As usual, we identify the vertices and their images, as well as the edges and the line segments representing them.

A \emph{crossing} in a rectilinear drawing $D$ of $G$ is a common interior point of at least two edges of $D$ where they properly cross.
A \emph{heavy crossing} in $D$ is a common interior point of at least three edges of $D$ where they properly cross.
We say that $D$ is \emph{generic} if there are no heavy crossings in $D$.
That is, crossings in a generic drawing $D$ are the points where exactly two edges of $D$ cross.

We focus on rectilinear drawings of complete graphs $K_n$ on $n$ vertices.
We say that a rectilinear drawing $D$ of a graph $K_n$ is \emph{convex} if the points representing the vertices of $K_n$ are in convex position.
We say that a convex drawing $D$ of $K_n$ is \emph{regular} if the points representing the vertices of $K_n$ form a regular $n$-gon; see Figure~\ref{fig-K8} for regular drawings of $K_8$ and $K_{12}$.
\begin {figure}[ht]
  \centering
  \includegraphics[width=\textwidth]{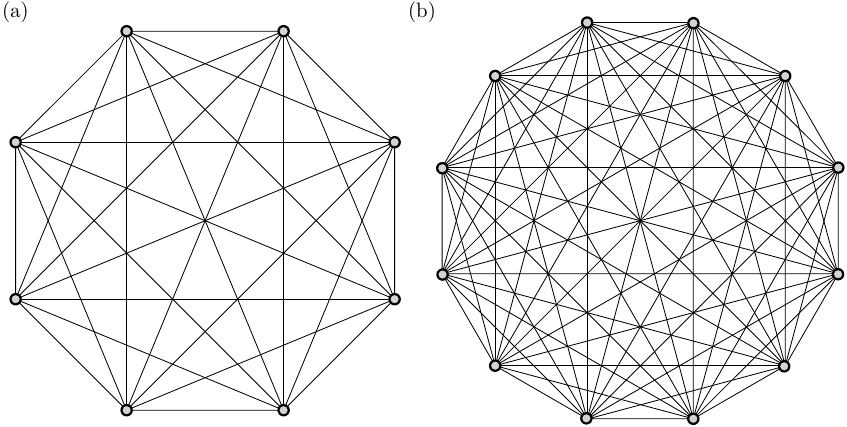}
  \caption {Regular drawings of $K_8$ (part~(a)) and $K_{12}$ (part~(b)). Observe that none of these drawings contains a 5-face.}
  \label{fig-K8}
\end {figure}

A \emph{face} in a rectilinear drawing $D$ of $K_n$ is a non-empty connected component of $\mathbb{R}^2\setminus D$.
Note that exactly one face of $D$ is unbounded and that every bounded face of $D$ is a convex polygon.
Thus, we can define the \emph{size} of a bounded face $F$ of $D$ to be the number of vertices of the polygon that forms $F$.
If the size of $F$ equals $k$, then we call $F$ a \emph{$k$-face} of $D$.

In this paper, we study extremal problems about the bounded faces of a given size in convex drawings of $K_n$.
To our knowledge, there has been no systematic study of this topic despite the fact that it offers an abundance of natural and interesting problems.
For example, what is the largest face we can always find in a convex drawing of $K_n$ for large $n$?
What if we restrict ourselves to generic convex drawings of $K_n$?
Or to regular drawings of $K_n$?
In this paper, we address these questions and we pose several natural open problems.

\section{Previous Work}

Even though these problems are very natural and that rectilinear drawings of $K_n$ have been studied extensively, we did not find any relevant reference in the literature.
The existence of faces of a given size in regular drawings of $K_n$ was recently considered by Shannon and Sloane~\cite{oeis}, who computed the values from Table~\ref{tab-OEIS}, but we are not aware of any publication.
The total number of faces in a regular drawing of $K_n$ was considered by Harborth~\cite{harborth69} and Poonen and Rubinstein~\cite{poRu98}, but these results do not distinguish faces of different sizes and do not apply to all convex drawings of $K_n$.
Finally, Hall~\cite{hall04} studied large faces in convex drawings of $K_n$ where the vertices are points from the integer lattice.

Concerning other graph classes, Griffiths~\cite{griff10} calculated the number of regions enclosed by the edges of so-called regular drawings of the complete bipartite graphs $K_{n,n}$.
There are also various results about the complexity of faces in the more general setting of line arrangements; for example~\cite{ahkm95,aegs92,bcv18,furPal84,matVal97}.
However, we do not know any result that would imply the existence of large bounded faces in all convex drawings of sufficiently large $K_n$.

Closely related to our paper is the work of Poonen and Rubinstein~\cite{poRu98} who gave a formula for the number of crossings in regular drawings of~$K_n$ and used it to count the number of faces in regular drawings of~$K_n$.
In particular, it follows from their formula that all regular drawings of $K_n$ with odd $n$ have $\binom{n}{4}$ crossings and thus are generic.
They also showed that, apart from the center, no crossing is the intersection of more than 7 edges of a regular drawing of $K_n$  for any positive integer $n$.
We also note that these results are connected to the well-known \emph{Blocking conjecture}~\cite{mat09,porWood10}, which states that the minimum number of points that block visibilities between $n$ points in the plane in general position grows superlinearly in~$n$.
The work of Poonen and Rubinstein~\cite{poRu98} implies that regular drawings of $K_n$ cannot be used for placing a small number of blocking points since their number for regular drawings of $K_n$ is quadratic~\cite{mat09}.

\section{Our Results}

First, we address the question about the maximum size of a face that we can always find in convex or regular drawings of $K_n$ for large $n$.
We observe that finding faces of size 3 or 4 in convex drawings of $K_n$ is not difficult.

\begin{proposition}
\label{prop-3-4faces}
Let $n$ be a positive integer and $D$ a
convex drawing of $K_n$.
Then, $D$ contains a 3-face if and only if $n \geq 3$.
Moreover, $D$ contains a 4-face if and only if $n \geq  6$.
\end{proposition}

To find larger faces, we restrict ourselves to generic convex drawings of $K_n$.
In this case, we can show that a 5-face always exists if we have at least five vertices.

\begin{theorem}
\label{thm-5faceGeneric}
For every positive integer $n$ and every generic convex drawing $D$ of~$K_n$, the drawing $D$ contains a 5-face if and only if $n \geq 5$.
\end{theorem}

On the other hand, we can provide examples of generic convex drawings of $K_n$ with arbitrarily large $n$ that do not contain any $k$-face with $k \geq 6$.

\begin{theorem}
\label{thm-No6faceGeneric}
For every positive integer $n$, there is a generic convex drawing of $K_n$ that does not contain any $k$-face with $k \geq 6$.
\end{theorem}

Thus, in the case of generic convex drawings of $K_n$, we can settle the question about the largest face we can always find completely.
A $k$-face with $k \in\{3,4,5\}$ is guaranteed in all sufficiently large drawings, while faces of sizes larger than 5 can be avoided (even simultaneously).
The problem, however, becomes significantly more difficult if we allow heavy crossings.

We were not able to find a $k$-face with $k \geq 5$ in every sufficiently large convex drawing of~$K_n$.
In fact, finding larger faces becomes surprisingly difficult already for regular drawings of~$K_n$.
Here, however, we can at least show that a 5-face always exists in all sufficiently large regular drawings of $K_n$.
In fact, we can even precisely characterize the values of $n$ for which a regular drawing of $K_n$ contains a 5-face.

\begin{theorem}
\label{thm-5faceRegular}
For a positive integer $n$, a regular drawing of $K_n$ contains a 5-face if and only if $n \notin \{1,2,3,4,6,8,12\}$.
\end{theorem}

The proof of Theorem~\ref{thm-5faceRegular} is quite involved and is based on the results obtained by Poonen and Rubinstein~\cite{poRu98}.

Finally, although we were not able to find a 5-face in all sufficiently large convex drawings of $K_n$, we can at least show that every convex drawing of $K_7$ contains at least one.

\begin{proposition}
\label{prop-K7}
Every convex drawing of $K_7$ contains a 5-face.
\end{proposition}

The remainder of the paper is organized as follows.
We first prove Proposition~\ref{prop-3-4faces} in Section~\ref{sec-3-4faces}.
Then, Theorems~\ref{thm-5faceGeneric} and~\ref{thm-No6faceGeneric} are proved in Sections~\ref{sec-5faceGeneric} and~\ref{sec-No6faceGeneric}, respectively.
The proofs of Theorem~\ref{thm-5faceRegular} and Proposition~\ref{prop-K7} can be found in Sections~\ref{sec-5faceRegular} and~\ref{sec-K7}, respectively.
Finally, we state some open problem and possible directions for future research in Section~\ref{sec-openProblems}.

\section{Finding 3- and 4-faces in convex drawings}
\label{sec-3-4faces}

For a positive integer $n$, let $D$ be a convex drawing of $K_n$.
We show that $D$ contains a 3-face if and only if $n \geq 3$ and that $D$ contains a 4-face if and only if $n \geq  6$.

Since $D$ is convex, the boundary of its convex hull is a convex $n$-gon with vertices formed by the points representing the vertices of $K_n$.
Let $v_1,\dots,v_n$ be the vertices of $D$ traced in this order along the boundary of the convex hull in the, say, clockwise direction.

Obviously, $D$ does not contain a 3-face if $n \leq 2$.
Now, if $n \geq 3$, then it suffices to consider the vertex $v_2$ of $D$.
The face $F$ of $D$ bounded by $v_1v_2$, $v_1v_3$ and $v_2v_n$ is a 3-face as there are no edges of $D$ that can intersect $F$; see part~(a) of Figure~\ref{fig-3-4faces}.

\begin {figure}[ht]
  \centering
  \includegraphics[width=\textwidth]{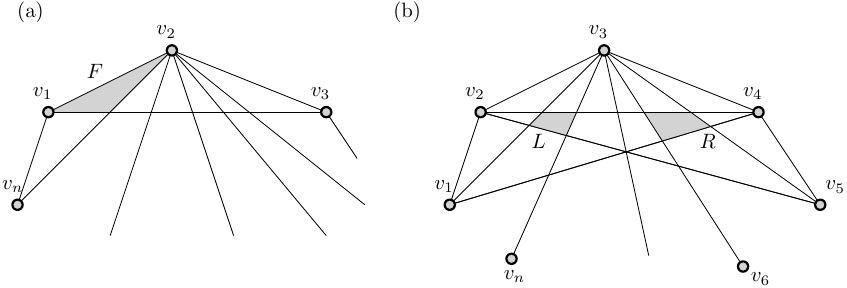}
  \caption {An illustration of the proof of Proposition~\ref{prop-3-4faces}. Finding a 3-face (part~(a)) and a 4-face (part~(b)).}
  \label{fig-3-4faces}
\end {figure}

It is easy to check that no convex drawing of $K_n$ contains a 4-face if $n\leq 5$ as each such drawing is essentially unique.
Now, assume $n \geq 6$.
Consider the drawing $D'$ induced by vertices $v_1,\dots,v_5$ in~$D$.
Let $F$ be the 5-face of $D'$ bounded by the line segments $v_iv_{i+2}$ with $i \in \{1,\dots,5\}$ (indices taken modulo 5).
Note that $F$ is not incident to any vertex of $D$.
Since $n \geq 6$, we have $v_n \notin \{v_1,\dots,v_5\}$.
The line segment $v_3v_n$ then splits $F$ into two parts where the left part $L$ is a face of $D$; see part~(b) of Figure~\ref{fig-3-4faces}.
Similarly, the line segment $v_3v_6$ splits $F$ into two parts where the right part $R$ is a face of $D$.
It now suffices to observe that at least one of the faces $L$ and $R$ is a 4-face in $D$ as every line segment $v_3v_i$ with $i \in \{6,7,\dots,n\}$ splits $F$ into two parts, where at least one is a convex 4-gon.

\section{Finding 5-faces in generic convex drawings}
\label{sec-5faceGeneric}
For a positive integer $n$, let $D$ be a generic convex drawing of $K_n$.
We show that $D$ contains a 5-face if and only if $n\geq 5$.

Clearly, $D$ does not contain a 5-face if $n \leq 4$.
Now, if $n=5$, then it is easy to see that there is always a 5-face in $D$ that is not incident to any vertex of $D$.
So we assume that $n \geq 6$.
Let $v_1,\dots,v_n$ be the vertices of $D$ traced in this order along the boundary of the convex hull in the, say, clockwise direction.

\begin {figure}[ht]
  \centering
  \includegraphics[width=\textwidth]{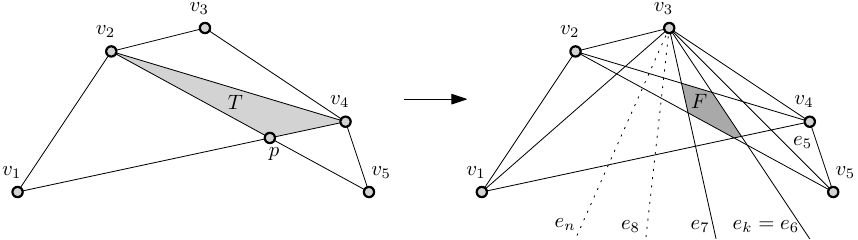}
  \caption {An illustration of the proof of Theorem~\ref{thm-5faceGeneric}.}
  \label{fig-5faceGeneric}
\end {figure}

Let $p=v_1v_4\cap v_2v_5$ and consider the triangle $T=pv_2v_4$; see Figure~\ref{fig-5faceGeneric}.
This triangle is intersected by precisely the edges $e_i=v_3v_i$ for $i \in \{5,\dots,n\} \cup \{1\}$. 
Since the drawing $D$ is generic, no edge $e_i$ passes through $p$, and, moreover, both sides $pv_4$ and $pv_2$ of the triangle $T$ are crossed at least once (by $e_5$ and $e_1$, respectively). 
Let $k$ be the maximum integer from $\{5,\dots,n\}$ such that $e_k$ crosses $pv_4$. Then $e_k$, $e_{k\pmod n+1}$, and the three sides of $T$ determine a 5-face $F$ (with one vertex $p$).

\section{Generic convex drawings with no 6-faces}
\label{sec-No6faceGeneric}

We prove that, for every positive integer $n$, there is a generic convex drawing of~$K_n$ that does not contain a $k$-face with $k \geq 6$.
We apply a similar construction to the one used by Balko et al.~\cite{bcgghvw22}.
It is also the planar case of the construction used by Bukh, Matou\v{s}ek, and Nivasch~\cite{bmn10}.

First, we state some auxiliary definitions.
For an integer $k \geq 3$, a set of $k$ points in the plane is a \emph{$k$-cup} if all its points lie on the graph of a convex function.
Similarly, a set of $k$ points is a \emph{$k$-cap} if all its points lie on the graph of a concave function.
Clearly, $k$-cups and $k$-caps are sets of points in convex position.
A convex polygon $P$ is \emph{$k$-cap free} if no $k$ vertices of $P$ form a $k$--cap.
Note that $P$ is $k$-cap free if and only if it is bounded from above by at most $k-2$ segments (edges of $P$).
Analogously, $P$ is \emph{$k$-cup free} if no $k$ vertices of $P$ form a $k$--cup.
Observe that vertices of a $k$-face determine an $a$-cap and a $u$-cup that share the leftmost and the rightmost vertex and satisfy $a+u = k+2$.
We use $e(P)$ to denote the leftmost edge bounding $P$ from above; see part~(a) of Figure~\ref{fig-drawing}.

\begin{figure}
    \centering
    \includegraphics{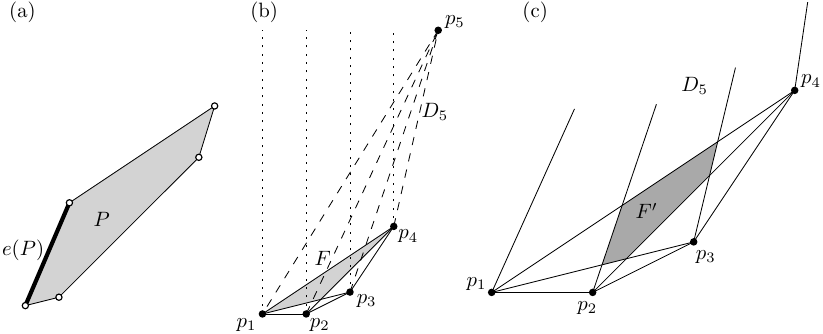}
    \caption{(a) A 4-cap free and 5-cup free polygon $P$ that is not 3-cap free nor 4-cup free. (b) An example of constructing the generic convex drawing $D_n$ for $n=5$. If the point $p_n$ is chosen sufficiently high above $V(D_{n-1})$, then each line segment $\overline{p_i p_n}$ with $i < n$ is very close to the vertical line containing $p_i$ and thus all faces of $D_n$ will be 4-cap free and 5-cup free. (c) The face $F$ of $D_{n-1}$ is split into new faces of~$D_n$ and contains the face $F'$ that is 4-cap free and 5-cup free but not 3-cap free nor 4-cup free.}
    \label{fig-drawing}
\end{figure}

We inductively construct a certain generic convex drawing $D_n$ of $K_n$ with vertices represented by points $p_1,\dots,p_n$ that form an $n$-cup in the plane and their $x$-coordinates satisfy $x(p_i) = i$; see part~(b) of Figure~\ref{fig-drawing}.
Let $V(D_n)$ denote the vertex set of $D_n$. We recall that we identify the vertices of $K_n$ and the points from $D_n$ representing them.
We let  $V(D_1)=\{(1,0)\}$ and $V(D_2)=\{(1,0),(2,0)\}$.
Now, assume that we have already constructed the drawing $D_{n-1}$ with $V(D_{n-1}) = \{p_1,\dots,p_{n-1}\}$ for some integer $n \geq 3$.
We choose a sufficiently large number $y_n$, and we let $p_n$ be the point $(n,y_n)$.
We then set $V(D_n) = V(D_{n-1}) \cup \{p_n\}$ and we let $D_n$ be the drawing of $K_n$ on this vertex set.
The number $y_n$ is chosen large enough so that the following three conditions are satisfied:
\begin{enumerate}
\item for every $i=1,\dots,n-1$, each intersection point of two line segments spanned by points from $V(D_{n-1})$ lies on the left side of the line $\overline{p_ip_n}$ if and only if it lies to the left of the vertical line $x=i$ containing the point $p_i$, 

\item if $F$ is a 4-cap free face of $D_n$ that is not 3-cap free, then there is no point $p_i$ below the (relative) interior of $e(F)$,

\item no crossing of two edges of $D_n$ lies on the vertical line containing some point~$p_i$.
\end{enumerate}

Choosing the point $p_n$ is indeed possible as for a sufficiently large $y$-coordinate $y_n$ of $p_n$ we get that for each $i$, all the intersections of the line segments $p_ip_n$ with line segments of $D_{n-1}$ lie very close to the vertical line $x=i$ containing the point~$p_i$.
Observe that no line segment of $D_n$ is vertical and that there are no heavy crossings in $D_n$.
Since the points $p_1,\dots,p_n$ form an $n$-cup, they are in convex position and $D_n$ is a generic convex drawing of~$K_n$.

It remains to prove that there are no $k$-faces with $k \geq 6$ in $D$.
To show that, we use the following lemma.

\begin{lemma}
  \label{lem-drawing-faces}
Each bounded face of $D_n$ is a 4-cap free and 5-cup free convex polygon.
\end{lemma}
\begin{proof}
We prove both claims by induction on $n$.
Both statements are trivial for $n \leq 3$ so assume $n \geq 4$.
Now, let $F$ be a bounded face of $D_{n-1}$.

We first show that all faces of $D_n$ contained in $F$ are 4-cap free.
By the induction hypothesis, $F$ is a 4-cap free convex polygon.
If $F$ is 3-cap free, then, by the choice of $p_n$, the line segments $\overline{p_i p_n}$ 
split $F$ into 4-cap free polygons (with the leftmost one being actually 3-cap free); see part~(c) of Figure~\ref{fig-drawing}.
If $F$ is 4-cap free and not 3-cap free, then the choice of $p_n$ guarantees that the line segments $\overline{p_i p_n}$ split $F$ into 4-cap free polygons.
This is because the leftmost such polygon contains the whole edge $e(F)$ as there is no $p_i$ below the edge~$e(F)$.

Now, we prove that all faces of $D_n$ contained in $F$ are 5-cup free.
If $F$ is 4-cup free, then the choice of $p_n$ implies that the line segments $\overline{p_i p_n}$ 
split $F$ into 5-cup free polygons; see part~(c) of Figure~\ref{fig-drawing}.
Now assume that $F$ is 5-cup free and not 4-cup free.
Suppose for contradiction that $F$ contains a bounded face $F'$ of $D_n$ that is not 5-cup free.
Then, $F'$ is obtained from $F$ by some line segment $p_ip_n$ splitting the rightmost edge $e$ of the lower envelope of $F$.
By construction of $D_n$, the point $p_i$ then lies below the (relative) interior of $e$.
In particular, there is a face $F''$ of $D_{n-1}$ that shares $e$ with $F$.
The rightmost vertex of $F$ cannot be a vertex of $D_n$ as all faces in $D_{n-1}$ with $p_i$ as their rightmost vertex are 3-faces, which are 4-cup free.
Thus, the rightmost vertex of $F$ is a crossing in $D_{n-1}$.
Then, we have $e(F'')=e$.
However, this contradicts the second property from the construction of~$D_n$ as $p_i$ lies below $e(F'')$.
Thus, all faces of $D_n$ contained in $F$ are 5-cup free.

It remains to consider the inner faces of $D_n$ which lie outside of the convex hull of the points $p_1,\dots,p_{n-1}$. These faces lie inside the triangle $p_1 p_{n-1} p_n$.
They are all triangular and therefore are 3-cap free and also 4-cup free; see the three faces with the topmost vertex $p_n=p_5$ in part~(c) of Figure~\ref{fig-drawing}.
\end{proof}

Now, suppose for contradiction that there is a $k$-face $F$ in $D_n$ for some integer $k \geq 6$.
By Lemma~\ref{lem-drawing-faces}, the face $F$ is a 4-cap free and 5-cup free convex polygon.
On the other hand, the vertex set of $F$ is in convex position and thus determines an $a$-cap and a $u$-cup that share the leftmost and the rightmost vertex and satisfy $a+u \geq 8$.
Therefore, we either have $a \geq 4$ or $u \geq 5$. However, this contradicts the fact that $F$ is 4-cap free and 5-cup free.

\section{Finding 5-faces in regular drawings}
\label{sec-5faceRegular}

Let $n$ be a positive integer and let $D$ be a regular drawing of $K_n$.
We show that $D$ contains a 5-face if and only if $n \notin \{1,2,3,4,6,8,12\}$.

\subsection{Preliminaries}

To find a 5-face in $D$, we will need to analyze heavy crossings in $D$.
We do so by using methods applied by Poonen and Rubinstein~\cite{poRu98} that we now briefly describe.

\begin {figure}[ht]
  \centering
  \includegraphics{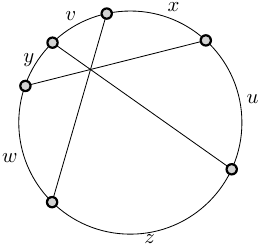}
  \caption {A heavy crossing and its parameters.}
  \label{fig-heavyIntersection}
\end {figure}

We assume without loss of generality that $D$ has vertices on a unit circle centered in the origin.
Consider three pairwise crossing edges with endpoints on a unit circle dividing up the circumference into arc lengths $u$, $x$, $v$, $y$, $w$, and $z$; see Figure~\ref{fig-heavyIntersection}.
It is not difficult to show that three such edges form a heavy crossing if and only if
\[\sin(u/2) \sin(v/2) \sin(w/2) = \sin(x/2) \sin(y/2) \sin(z/2);\]
see, for example,~\cite{poRu98}.
Thus, to determine when three edges form a heavy crossing, we need to characterize the positive rational solutions to the following system of equations:
\begin{equation}
\label{eq-heavyIntersections}
\begin{split}
\sin(\pi U)\sin(\pi V)\sin(\pi W) &= \sin(\pi X)\sin(\pi Y)\sin(\pi Z)\\
U+V+W+X+Y+Z &= 1,
\end{split}
\end{equation}
where we substitute $U:=u/(2\pi)$, $V:=v/(2\pi)$, $W:=w/(2\pi)$, $X:=x/(2\pi)$, $Y:=y/(2\pi)$, and $Z:=z/(2\pi)$.
Using a theory of trigonometric Diophantine equations~\cite{conJon76,mann65}, Poonen and Rubinstein~\cite{poRu98} classified the solutions of~\eqref{eq-heavyIntersections}.

\begin{theorem}[Theorem~4 in~\cite{poRu98}]
\label{thm-PoRu}
The positive rational solutions to~\eqref{eq-heavyIntersections}, up to symmetry, can be classified as follows:
\begin{enumerate}
    \item The trivial solutions that satisfy $U+V+W = 1/2$ and $X,Y,Z$ are a permutation of $U,V,W$.
    \item Four one-parameter families of solutions, listed in Table~\ref{tab-solutions}.
    \item Sixty-five “sporadic” solutions, listed in Table 4 in~\cite{poRu98}.
\end{enumerate}
\end{theorem}

\begin{table}
\centering
\begin{tabular}{c|c|c|c|c|c|c|c}
Type & $U$ & $V$ & $W$ & $X$ & $Y$ & $Z$ & Range \\
\hline
\hline
I & 1/6 & $t$ & $1/3-2t$ & $1/3+t$ & $t$ & $1/6-t$ & $0<t<1/6$ \\
\hline
II & 1/6 & $1/2-3t$ & $t$ & $1/6-t$ & $2t$ & $1/6+t$ & $0<t<1/6$ \\
\hline
III & 1/6 & $1/6-2t$ & $2t$ & $1/6-2t$ & $t$ & $1/2+t$ & $0<t<1/12$ \\
\hline
IV & $1/3-4t$ & $t$ & $1/3+t$ & $1/6-2t$ & $3t$ & $1/6+t$ & $0<t<1/12$
\end{tabular}
\caption{The nontrivial infinite families of solutions to~\eqref{eq-heavyIntersections} (taken from~\cite{poRu98}).}
\label{tab-solutions}
\end{table}

\subsection{Finding a 5-Face}

First, the statement of Theorem~\ref{thm-5faceRegular} is trivial for $n \leq 4$ so we assume $n \geq 5$.
Second, it follows from the formula on the number of crossings in $D$ by Poonen and Rubinstein~\cite{poRu98} that if $n$ is odd, then $D$ is generic.
In this case, there is a 5-face in~$D$ by Theorem~\ref{thm-5faceGeneric}.
We can thus assume that $n$ is even.
If $n = 6$, then it is easy to verify that there is no 5-face in the regular drawing of $K_n$.
Thus, from now on, we assume that $n$ is even and $n \geq 8$ and we show that if $n \notin \{8,12\}$, then there is a 5-face in $D$. 

To find a 5-face in $D$, we consider the following configurations that appear in regular drawings of $K_n$ with $n \geq 8$.
Let $v_1,\dots,v_n$ be the vertices of $D$ traced in this order along the boundary of the convex hull in the clockwise direction.
We assume without loss of generality that $D$ is rotated so that the line segment $v_2v_3$ is horizontal.
Let $a$ be an integer such that $0 \leq a \leq n/2 - 4$.
Let $R_a$ be the region bounded by the line segments $v_2v_{5+a}$, $v_2v_{6+a}$, $v_3v_{n-a-1}$, and $v_3v_{n-a}$; see part~(a) of Figure~\ref{fig-5faceRegular}.
Observe that $R_a$ is a convex 4-gon with two vertices lying on a horizontal line.
We will show that if $n \notin \{8,12,18\}$, then there is $a$ such that $R_a$ contains a 5-face in $D$.

\begin {figure}[ht]
  \centering
  \includegraphics[width=\textwidth]{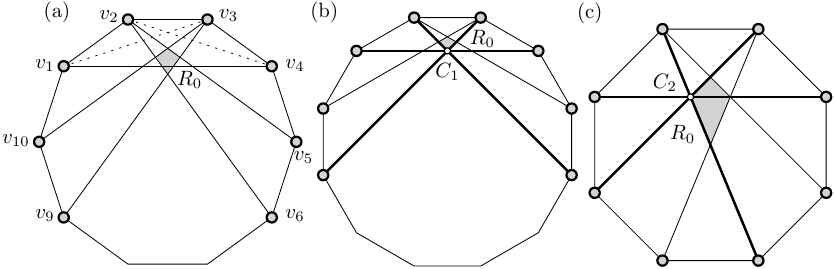}
  \caption {An illustration of the proof of Theorem~\ref{thm-5faceRegular}. (a) Here, $n=10$ and we have $a = 0$. Note that the region $R_a = R_0$ (denoted by light grey) contains a 5-face of~$D$ as $v_1v_4$ does not contain the bottom-most vertex of $R_0$. (b) If $n=12$, then the region $R_0$ does not contain a 5-face of $D$ as the bottom-most vertex of $R_0$ forms a heavy crossing in $D$. (c) If $n=8$, then the region $R_0$ also does not contain a 5-face of~$D$ as $v_1v_4$ contains two vertices of $R_0$ as two middle vertices of $R_0$ form heavy crossings in $D$. A similar situation happens with $R_1$ in the regular drawing of $K_{18}$.}
  \label{fig-5faceRegular}
\end {figure}

Note that the regions $R_a$ with $a \in \{0,1,\dots n/2-4\}$ form a vertical chain where the bottom-most point of $R_i$ is the topmost point of $R_{i+1}$.
Thus, there is an integer $b$ from $\{0,1,\dots n/2-4\}$ such that the line segment $v_1v_4$ intersects $R_b$.
If there are two choices for $b$, then we take the smaller one.
We will show that $R_b$ contains a 5-face of~$D$.

We show that the bottom-most point $r$ of $R_b$ does not lie below the crossing $s$ of the line segments $v_1v_{5}$ and $v_4v_n$; see part~(a) of Figure~\ref{fig-RS}.
First, we assume that $n \geq 12$ as the cases $n=8$ and $n=10$ can be easily checked; see Figures~\ref{fig-K8} and~\ref{fig-5faceRegular}, for example.
Now, let $m$ be the midpoint of $v_1v_4$.
It suffices to show that angle $\angle mv_2s$ is at least $\pi/n$; indeed, if $q$ is the top point of the region $R_b$, then the angle $\angle qv_2r$ is $\pi/n$ and since $q$ is above $m$ (by definition), we would conclude that $r$ is above $s$.
When $n=12$, we in fact have $\angle mv_2s=\pi/n$; see part~(b) of Figure~\ref{fig-RS}.
This is because line $v_2m$ passes through $v_6$ and line $v_2s$ passes through $v_7$ (both is easy to check by the criterion for 3 concurrent chords). 
In particular, the angle $\angle v_1v_2s$ is right, so points $v_1$, $v_2$, $m$, $s$ lie on a single circle.
Now assume $n>12$. 
We claim that in the 4-gon $v_1v_2ms$ both the angle at $v_2$ and the angle at $s$ become larger compared to the case $n=12$. 
In particular, $v_2$ lies strictly inside the circumcircle of triangle $v_1ms$, so the angles satisfy $\angle mv_2s > \angle mv_1s = \angle v_4v_1v_5 = \pi/n$ as desired.
To prove the claim, consider the two angles separately:
First, for the angle at $v_2$, we draw the new $n$-gon with $n>12$ on vertices $v'_1,\dots,v'_n$ such that it shares $v_2$ and $v_3$ with the 12-gon. 
Increasing $n$ makes the angle at $v_2$ larger, so $v_1$ moves clockwise along the circle with center $v_2$ and radius $v_2v_3=v_2v_1$ towards $v_1'$. 
Also, the midpoint $m$ moves up to the midpoint $m'$ of $v_1'v_4'$. Altogether, $\angle v_1'v_2m' > \angle v_1v_2m$.
Second, the angle at $s$ is simply $90^\circ - \angle v_4v_1v_5 = 90^\circ - \pi/n$, so it increases as $n$ increases.

\begin {figure}[ht]
  \centering
  \includegraphics{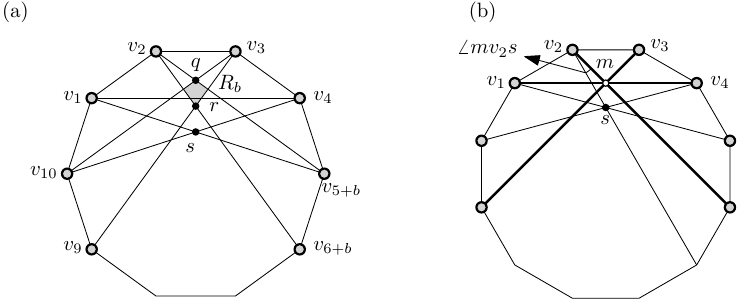}
  \caption {Proving that $r$ does not lie below $s$  for $n \geq 8$.}
  \label{fig-RS}
\end {figure}

Note that for $n > 8$, the interior of the line segment $v_4v_n$ contains $s$ and lies above the center of $D$, which is the bottom-most point of $R_{n/2-4}$.
Since $r$ is not below $s$, we have $b \leq n/2-5$ if $n>8$.

It also follows that the interior of $R_b$ is not intersected by any other edge of $D$ besides $v_1v_4$. 
Thus, if the line segment $v_1v_4$ does not contain any vertex of $R_b$, then we have a 5-face in $D$.
On the other hand, if the line segment $v_1v_4$ contains a vertex of $R_b$, then there would be no 5-face in $R_b$; see parts~(b) and~(c) of Figure~\ref{fig-5faceRegular}.

Therefore, it remains to show that $v_1v_4$ does not contain any vertex of $R_b$.
We apply Theorem~\ref{thm-PoRu} to prove that if $n \notin \{8,12,18\}$, then there is a 5-face of $D$ in $R_b$.
Observe that if the edge $v_1v_4$ contains a vertex of $R_b$ then, by our choice of $b$, it is either the bottom-most vertex of $R_b$ or the two middle vertices of $R_b$ that lie on a common horizontal line.
In each of these two cases, we get a heavy crossing, either the heavy crossing $C_1$ between edges $v_1v_4$, $v_2v_{6+b}$, and $v_3v_{n-b-1}$ or the heavy crossing $C_2$ between $v_1v_4$, $v_2v_{6+b}$, and $v_3v_{n-b}$.
These heavy crossings are denoted by thick black edges in parts~(b) and~(c) of Figure~\ref{fig-5faceRegular}.

\subsection{\texorpdfstring{Excluding $C_1$}{Excluding C1}}

First, we show that the heavy crossing $C_1$ can appear only if $n=12$.
The three edges determining $C_1$ divide up the circumference of the unit circle into arc lengths $u'$, $x'$, $v'$, $y'$, $w'$, and $z'$; see Figure~\ref{fig-C1}.
It follows that $u' = \frac{2\pi(b+2)}{n}$, $v' = \frac{2\pi}{n}$, and $w' = \frac{2\pi(b+2)}{n}$.
Similarly, the arc lengths $x'$, $y'$, and $z'$ are portions of the unit circle between vertices $v_3$ and $v_4$, between $v_1$ and $v_2$, and between $v_{6+b}$ and $v_{n-b-1}$, respectively.
We then obtain $x' = \frac{2\pi}{n}$, $y' = \frac{2\pi}{n}$, and $z' = \frac{2\pi(n-2b-7)}{n}$.
By substituting $U':=u'/(2\pi)$, $V':=v'/(2\pi)$, $W':=w'/(2\pi)$, $X':=x'/(2\pi)$, $Y':=y'/(2\pi)$, and $Z':=z'/(2\pi)$, we get 
\[S=(U',V',W',X',Y',Z') = \left(\frac{b+2}{n}, \frac{1}{n}, \frac{b+2}{n}, \frac{1}{n}, \frac{1}{n}, \frac{n-2b-7}{n}\right)\]
and it suffices to check that this 6-tuple $S$ is a solution of~\eqref{eq-heavyIntersections} if and only if $n=12$.

\begin {figure}[ht]
  \centering
  \includegraphics{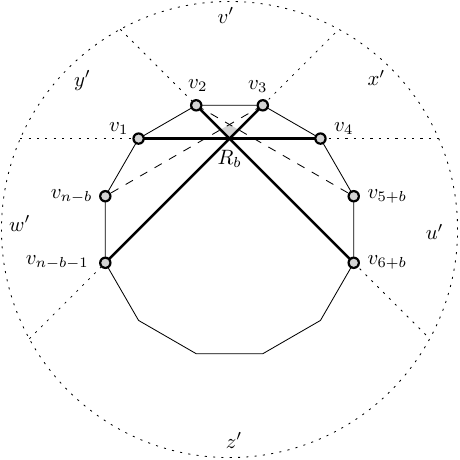}
  \caption {An example of a heavy crossing $C_1$ and its parameters.}
  \label{fig-C1}
\end {figure}

We have $V'=X'=Y'$ and $U' = W'$, so there are at most three distinct values among $U',V',W',X',Y',Z'$.
By checking Table 4 in~\cite{poRu98} for the 65 sporadic solutions of~\eqref{eq-heavyIntersections}, we can see that no sporadic solution satisfies these equalities.

The 6-tuple $S$ does not give a trivial solution of~\eqref{eq-heavyIntersections} as such a solution consists of two identical triples of numbers that sum up to $1/2$, which would imply that $n-2b-7=1$.
This is impossible as then $b = n/2-4$, which contradicts $b \leq n/2-5$. 

Thus, it suffices to check that if $n \neq 12$, then $S$ is not contained in any of the four infinite families of nontrivial solutions of~\eqref{eq-heavyIntersections} that are listed in Table~\ref{tab-solutions}.
We check each of these four families of solutions separately.

Suppose first that $U',V',W',X',Y',Z'$ attain the values $U,V,W,X,Y,Z$ (not necessarily in this order) from a solution of type I from Table~\ref{tab-solutions}.
Since $0<t<1/6$, we have $V=Y,Z<U<X$.
Since $V'=X'=Y'$, we obtain $t = 1/6-t = V=Y=Z=V'=X'=Y'=1/n$, which eventually gives a solution $S=(\frac{2}{12},\frac{1}{12},\frac{2}{12},\frac{1}{12},\frac{1}{12},\frac{5}{12})$ for $n=12$ and $b=0$.
This, however, is impossible since we assume $n \neq 12$.

Suppose now that $U',V',W',X',Y',Z'$ attain the values $U,V,W,X,Y,Z$ (not necessarily in this order) from a solution of type II from Table~\ref{tab-solutions}.
Since $0<t<1/6$, we have $W<U$, $Y<Z$, and $X<U<Z$.
Since $V'=X'=Y' \leq U'=W',Z'$, we have $t=1/6-t=W=X=V'=X'=Y'=1/n$, which gives $1/n=t=1/12$.
This, again, is impossible as $n \neq 12$.

Now, we consider type III.
Let $U',V',W',X',Y',Z'$ attain the values $U,V,W,\allowbreak X,Y,Z$ (not necessarily in this order) from a solution of type III from Table~\ref{tab-solutions}.
since $0<t<1/12$, we have $Y<W<U<Z$.
However, this is impossible as there are at most three distinct values in $U',V',W',X',Y',Z'$ .

Finally, suppose that $U',V',W',X',Y',Z'$ attain the values $U,V,W,X,Y,Z$ (not necessa\-rily in this order) from a solution of type IV from Table~\ref{tab-solutions}.
Since $0<t<1/12$, we have $V<Y<Z$ and $X<U<W$.
We immediately see that we cannot choose the three lowest values $V',X',Y'$ from $U,V,W,X,Y,Z$ to be the same.

Altogether, we see that if $n \neq 12$, there is no crossing $C_1$ in $D$.

\subsection{\texorpdfstring{Excluding $C_2$}{Excluding C2}}

Now, we prove that the heavy crossing $C_2$ can appear only if $n\in\{8,18\}$.
The three edges determining $C_2$ divide up the circumference of the unit circle into arc lengths $u'$, $x'$, $v'$, $y'$, $w'$, and $z'$; see Figure~\ref{fig-C2}.
The arc lengths $u'$, $v'$, and $w'$ are portions of the unit circle between vertices $v_4$ and $v_{6+b}$, between $v_2$ and $v_3$, and between $v_{n-b}$ and $v_1$, respectively.
It follows that $u' = \frac{2\pi(b+2)}{n}$, $v' = \frac{2\pi}{n}$, and $w' = \frac{2\pi(b+1)}{n}$.
Similarly, the arc lengths $x'$, $y'$, and $z'$ are portions of the unit circle between vertices $v_3$ and $v_4$, between $v_1$ and $v_2$, and between $v_{6+b}$ and $v_{n-b}$, respectively.
We then obtain $x' = \frac{2\pi}{n}$, $y' = \frac{2\pi}{n}$, and $z' = \frac{2(n-2b-6)\pi}{n}$.
By substituting $U':=u'/(2\pi)$, $V':=v'/(2\pi)$, $W':=w'/(2\pi)$, $X':=X'/(2\pi)$, $Y':=y'/(2\pi)$, and $Z':=z'/(2\pi)$, we get 
\[S=(U',V',W',X',Y',Z') = \left(\frac{b+2}{n}, \frac{1}{n}, \frac{b+1}{n}, \frac{1}{n}, \frac{1}{n}, \frac{n-2b-6}{n}\right)\]
and it suffices to check that this 6-tuple $S$ is a solution of~\eqref{eq-heavyIntersections} if and only if $n\in\{8,18\}$.

\begin {figure}[ht]
  \centering
  \includegraphics{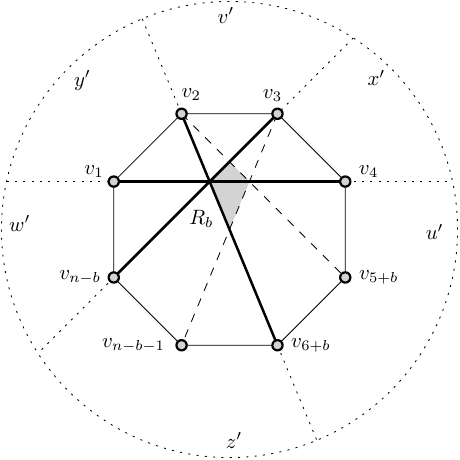}
  \caption {An example of a heavy crossing $C_2$ and its parameters.}
  \label{fig-C2}
\end {figure}

We have $V'=X'=Y'=U'-W' \leq Z'$, so there are at most four distinct values among $U',V',W',X',Y',Z'$.
Similarly as before, by checking Table 4 in~\cite{poRu98} for the 65 sporadic solutions of~\eqref{eq-heavyIntersections}, we note that there are only two sporadic solutions with one value appearing three times: $(\frac{1}{15},\frac{1}{15},\frac{7}{15},\frac{1}{15},\frac{1}{10},\frac{7}{30})$ and $(\frac{1}{30},\frac{1}{30},\frac{7}{10},\frac{1}{30},\frac{1}{15},\frac{2}{15})$.
Neither of these solutions, however, agrees with the values of $S$ for any $b$ and $n$.

Suppose that $S$ yields a trivial solution of~\eqref{eq-heavyIntersections}.
The fact that each trivial solution of~\eqref{eq-heavyIntersections} consists of two identical triples $T_1$ and $T_2$ of numbers that sum up to $1/2$ implies that the value $1/n$ has to appear in one of these triples at least twice, say, in $T_1$.
If it appears three times in $T_1$, then $T_2=(\frac{b+1}{n},\frac{b+2}{n},\frac{n-2b-6}{n})=(\frac{1}{n},\frac{1}{n},\frac{1}{n})$, which is impossible for any $b$ and $n$.
Thus, the value $1/n$ appears twice in $T_1$ and once in $T_2$.
If $T_1=(\frac{1}{n},\frac{1}{n},\frac{b+1}{n})$, then $T_2=(\frac{1}{n},\frac{b+2}{n},\frac{n-2b-6}{n})$ which is impossible as $b+2 \notin \{1,b+1\}$.
If $T_1=(\frac{1}{n},\frac{1}{n},\frac{b+2}{n})$, then $T_2=(\frac{1}{n},\frac{b+1}{n},\frac{n-2b-6}{n})$ and we have $b=0$.
Since the entries of $T_1$ should sum up to $1/2$, this leads to $n=8$, which is impossible, as we assumed $n \neq 8$.
In the last case, $T_1=(\frac{1}{n},\frac{1}{n},\frac{n-2b-6}{n})$ and $T_2=(\frac{1}{n},\frac{b+1}{n},\frac{b+2}{n})$, we again get $b=0$ and $n=8$.

Now, it suffices to check that if $n \notin \{8,18\}$, then $S$ is not contained in any of the four infinite families of nontrivial solutions of~\eqref{eq-heavyIntersections} that are listed in Table~\ref{tab-solutions}.
We again verify this by considering each of the four types I, II, III, and IV separately.

Suppose first that $U',V',W',X',Y',Z'$ attain the values $U,V,W,X,Y,Z$ (not necessarily in this order) from a solution of type I from Table~\ref{tab-solutions}.
Since $0<t<1/6$, we have $Z,V=Y < U < X$.
When combined with the fact $V'=X'=Y'=U'-W' \leq Z'$, we get $1/6-t=t=V=Y=Z=V'=X'=Y'$.
Thus, $t=1/12$ and $n=12$.
This yields $W=1/6$, $U=1/6$, $X=5/12$, which is impossible, as no two of these values differ by $1/12$ as $U'$ and $W'$ should.

Now, suppose that $U',V',W',X',Y',Z'$ attain the values $U,V,W,X,Y,Z$ (not necessarily in this order) from a solution of type II from Table~\ref{tab-solutions}.
Since $0<t<1/6$, we have $W<U,Y<Z$ and $X<U,V$.
From $V'=X'=Y'=U'-W' \leq Z'$, we get $t=1/6-t=W=X=V'=X'=Y'$, which again implies $t=1/12$ and $n=12$.
From this, we derive $U=1/6$, $V=1/4$, $Y=1/6$, and $Z=1/4$.
We see that only two values from $U,V,W,X,Y,Z$ equal $1/12$, not three as it should in $S$.

Suppose that $U',V',W',X',Y',Z'$ attain the values $U,V,W,X,Y,Z$ (not necessarily in this order) from a solution of type III from Table~\ref{tab-solutions}.
Since $0<t<1/12$, we have $Y<W<U<Z$ and $V=X<U$.
Thus, from $V'=X'=Y'=U'-W' \leq Z'$, we get $t=1/6-2t=V=X=Y=V'=X'=Y'$, which implies $t=1/18$ and $n=18$.
From this we obtain a solution $(\frac{3}{18},\frac{1}{18},\frac{2}{18},\frac{1}{18},\frac{1}{18},\frac{10}{18})$ of~\eqref{eq-heavyIntersections}, which corresponds to $S$ for $n=18$ and $b=1$.
However, we assumed $n \neq 18$.

Finally, suppose that $U',V',W',X',Y',Z'$ attain the values $U,V,W,X,Y,Z$ (not necessa\-rily in this order) from a solution of type IV from Table~\ref{tab-solutions}.
Since $0<t<1/12$, we have $V<Y<Z$ and $X<U<W$.
We immediately see that we cannot choose the three lowest values $V',X',Y'$ from $U,V,W,X,Y,Z$ to be the same.

Altogether, we see that if $n \notin \{8,18\}$, there is no crossing $C_2$ in $D$.

\subsection{Remaining Drawings}

By now, we know that there is a 5-face in each regular drawing of $K_n$ with $n \notin \{1,2,3,4,8,12,\allowbreak 18\}$.
It is not difficult to check that regular drawings of $K_8$ and $K_{12}$ do not contain a 5-face; see Figure~\ref{fig-K8}.
Note that there are 8 heavy crossings in the regular drawing of $K_8$ corresponding to the trivial solution $(\frac{1}{8},\frac{2}{8},\frac{1}{8},\frac{1}{8},\frac{1}{8},\frac{2}{8})$ of~\eqref{eq-heavyIntersections}.
In the regular drawing of $K_{12}$, we have 73 heavy crossings, corresponding to the solutions $(\frac{1}{12},\frac{1}{12},\frac{4}{12},\frac{1}{12},\frac{1}{12},\frac{4}{12})$, $(\frac{2}{12},\frac{1}{12},\frac{2}{12},\frac{1}{12},\frac{1}{12},\frac{5}{12})$, $(\frac{2}{12},\frac{1}{12},\frac{3}{12},\frac{1}{12},\frac{2}{12},\frac{3}{12})$ of~\eqref{eq-heavyIntersections} and solutions obtained by permutations of these coordinates.

The situation is different for $K_{18}$ as, although our method does not find a 5-face there, this drawing contains one; see Figure~\ref{fig-K18}.
This is indeed a 5-face as one vertex of this region is formed by a heavy intersection corresponding to the trivial solution $(\frac{3}{18},\frac{2}{18},\frac{4}{18},\frac{2}{8},\frac{3}{18},\frac{4}{18})$ of~\eqref{eq-heavyIntersections}.
This finishes the proof of Theorem~\ref{thm-5faceRegular}.

\begin {figure}[ht]
  \centering
  \includegraphics{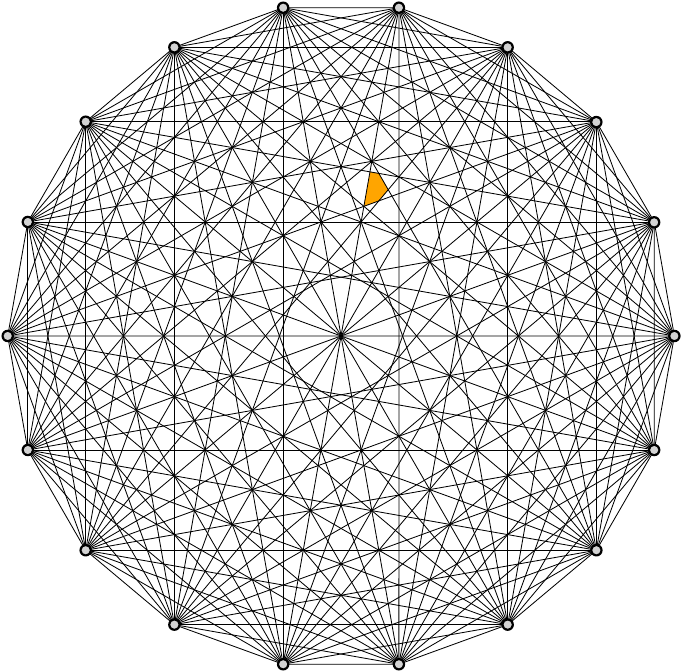}
  \caption {The regular drawing of $K_{18}$ with a distinguished 5-face (orange).}
  \label{fig-K18}
\end {figure}

\section{\texorpdfstring{Finding 5-faces in convex drawings of $K_7$}{Finding 5-faces in convex drawings of K7}}

\label{sec-K7}
Here, we show that every convex drawing of $K_7$ contains a 5-face.

Let $D$ be a convex drawing of $K_7$ on vertices $v_1,\dots,v_7$ that appear in this order along the boundary of the convex hull of $D$ when traced in the clockwise direction.

Fix $i \in \{1,\dots,7\}$.
Similarly as in the proof of Theorem~\ref{thm-5faceGeneric}, there is a 5-face $F_i$ in the drawing induced by vertices $\{v_i,\dots,v_{i+4}\}$ (from now on, indices are taken modulo 7) that is not incident to any vertex of $D$.
The face $F_i$ is split by edges $v_{i+2}v_{i+5}$ and $v_{i+2}v_{i+6}$ in $D$; see part~(a) of Figure~\ref{fig-K7}.
Let $p_i$ be the intersection point of $v_iv_{i+3}$ and $v_{i+1}v_{i+4}$.

\begin {figure}[ht]
  \centering
  \includegraphics{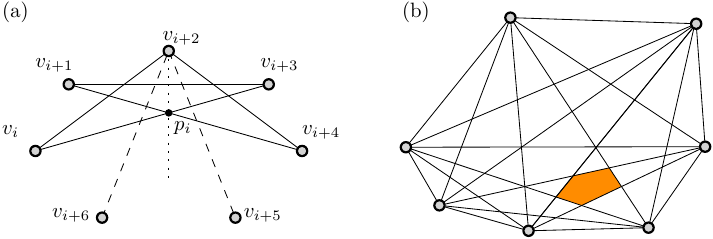}
  \caption {(a) Finding a 5-face in a convex drawing of $K_7$. (b) A convex drawing of $K_7$ with only one 5-face.}
  \label{fig-K7}
\end {figure}

If neither of the edges $v_{i+2}v_{i+5}$ and $v_{i+2}v_{i+6}$ passes through $p_i$, then we have a 5-face in $D$ with the vertex $p_i$.
Suppose otherwise, that is, one of the triples $(v_iv_{i+3},v_{i+1}v_{i+4},v_{i+2}v_{i+5})$ or $(v_{i+2}v_{i+6},v_iv_{i+3},v_{i+1}v_{i+4}) = (v_{i-1}v_{i+2},v_iv_{i+3},v_{i+1}v_{i+4})$ contains three concurrent edges.

Let $T = \{t_i = (v_iv_{i+3},v_{i+1}v_{i+4},v_{i+2}v_{i+5}) \colon i \in \{1,\dots,7\}\}$ be the set of triples of consecutive diagonals in $D$ that skip two vertices.
Observe that the elements $t_1,\dots,t_7$ of $T$ are cyclically ordered.

Suppose for contradiction that $D$ contains no 5-face.
Then, we just derived that either $t_{i-1}$ or $t_i$ are concurrent for every $i \in \{1,\dots,7\}$.
That is, either edges from $t_{i-1}$ or from $t_i$ determine a heavy crossing.
Since 7 is odd, at least two consecutive triples from $T$ determine a heavy crossing.
Since any two consecutive triples from $T$ share two edges, it follows that there are 4 diagonals from $T$ that all cross in a single point.
However, this is impossible, because in every 4-tuple of diagonals of the form $v_jv_{j+3}$ at least two of them share an endpoint as there are only 7 vertices, a contradiction.

Thus, there is at least one 5-face in every convex drawing of $K_7$. 
We remark that there are convex drawings of $K_7$ that contain a unique 5-face (see part~(b) of Figure~\ref{fig-K7}) while the regular drawing of $K_7$ contains seven 5-faces.

\section{Open Problems and Discussion}
\label{sec-openProblems}

The study of extremal questions about faces of a given size in convex drawings of~$K_n$ offers plenty of interesting and natural problems.
Here, we draw attention to some of them.

Although we were able to determine the largest size of a bounded face that appears in every sufficiently large generic convex drawing of $K_n$, the same question remains unsolved for all convex drawings of $K_n$.
In particular, the following problem is open.

\begin{problem}
\label{prob-5FaceConvex}
Is there a positive integer $n_0$ such that for every $n \geq n_0$ every convex drawing of $K_n$ contains a 5-face?
\end{problem}

Since the regular drawing of $K_{12}$ does not contain a 5-face, we have $n_0 \geq 13$, if it exists.
An affirmative answer to Problem~\ref{prob-5FaceConvex} would imply that every sufficiently large \emph{regular} drawing of $K_n$ contains a 5-face, a fact that was quite difficult to prove; see the proof of Theorem~\ref{thm-5faceRegular}.

Considering the regular drawings of $K_n$, although we proved that all sufficiently large regular drawings of $K_n$ contain a 5-face, we do not know much about larger faces.
It seems plausible that we can find arbitrarily large faces in regular drawings of $K_n$ as $n$ grows.

\begin{problem}
\label{prob-largeFaces}
Is it true that for every integer $k \geq 3$ there is an integer $n(k)$ such that every regular drawing of $K_n$ with $n \geq n(k)$ contains a $k$-face?
\end{problem}

For every integer $k$ with $3 \leq k \leq 19$, Shannon and Sloane~\cite{oeis} computed the value $a(k)$, which is the smallest $n$ such that the regular drawing of $K_n$ contains a $k$-face; see Table~\ref{tab-OEIS}.
Note that even if $a(k)$ exists, $n(k)$ might not.
Those computations suggest that the answer to Problem~\ref{prob-largeFaces} might be positive.
In such a case, it would be interesting to determine the growth rate of $n(k)$ with respect to $k$.
It follows from Proposition~\ref{prop-3-4faces} and Theorem~\ref{thm-5faceRegular} that $n(3) = 3$, $n(4)=6$, and $n(5)=13$.
We encourage the reader to visit the website\footnote{\href{https://fklute.com/regularkn.html}{fklute.com/regularkn.html}} to see the regular drawings for themselves.

\begin{table}
\centering
\begin{tabular}{c|c|c|c|c|c|c|c|c}
$k$ & 4 & 6 & 8 & 10 & 12 & 14 & 16 & 18  \\
\hline
\hline
$a(k)$ & 6 & 9 & 13 & 29 & 40 & 43 & 212 & 231 
\end{tabular}
\caption{The values of $a(k)$, the smallest $n$ such that the regular drawing of $K_n$ contains a $k$-face, computed by Shannon and Sloane~\cite{oeis}. We omitted the trivial values $a(k)=k$ for odd values of $k$.}
\label{tab-OEIS}
\end{table}
For $k$ odd, we trivially have $a(k)=k$ as the regular drawing of $K_n$ with $n$ odd contains an $n$-face in the center.
It might be interesting to explore the size of the largest faces in such drawings if we exclude this $n$-face.

A more difficult version of Problem~\ref{prob-largeFaces} would be to determine, for a given $k \geq 3$, all values of $n$ such that every regular drawing of $K_n$ contains a $k$-face.

Another possible direction is to count the minimum number of $k$-faces in a convex drawing of $K_n$.
For example, regarding $3$-faces, it is simple to show that there are always at least $n(n-3)$ by considering the area of a convex drawing around its outer-face as long as $n\geq 3$, but what is the growth rate of the minimum number of 3-faces with respect to $n$?

\begin{problem}
\label{prob-4Faces}
What is the minimum number of 3-faces in a convex drawing of $K_n$? What if the drawing is generic or regular?
\end{problem}

In the whole paper, we focused on convex drawings.
The problems we considered can also be stated for all rectilinear drawings of $K_n$.
Here, we can show that every generic rectilinear drawing of $K_n$ with $n \geq 10$ contains a $k$-face with $k \geq 5$.
This follows easily since, by a result of Harborth~\cite{Harborth1978}, every set $P$ of at least 10 points in the plane without three collinear points contains a \emph{5-hole}, that is, a set $H$ of 5 points in convex position with no point of $P$ in the interior of the convex hull of $H$.
If we then apply this result on the vertex set of a generic rectilinear drawing of $K_n$ and use a similar reasoning as in the proof of Theorem~\ref{thm-5faceGeneric} on the drawing induced by the resulting 5-hole in $D$, then we find a bounded face of size at least 5 in~$D$.

Finally, we considered the problem of finding a bounded face of size exactly $k$ for a given integer $k$, but it also makes sense to consider more relaxed variants of the above problems where we want to find a bounded face of size at least $k$ for a given integer $k$.
In particular, this leads to the following potentially simpler variant of Problem~\ref{prob-5FaceConvex}.

\begin{problem}
\label{prob-AtLeast5FaceConvex}
Is there a positive integer $n_1$ such that for every $n \geq n_1$ every convex drawing of $K_n$ contains a bounded face of size at least 5?
\end{problem}

We note that a simple double-counting argument based on Euler's formula yields the existence of $k$-faces in generic convex drawings of $K_n$ with $k \geq 4$.
If we knew that there are many 3-faces in such drawings, then the argument gives the existence of $k$-faces with $k \geq 5$.
This also illustrates that some insight for Problem~\ref{prob-4Faces} might have consequences for our original questions.

 \paragraph{Acknowledgment}
 We would like to thank the organizers of the 18th European Research Week on Geometric Graphs (GG Week 2023) where this research was initiated.
 We are very grateful to Alexandra Weinberger and Daniel Perz for useful discussions and for their help with finalizing the paper.

\bibliography{bibliography}
\bibliographystyle{plainnat}

\end{document}